\documentclass[pdftex,12pt,reqno]{amsart}
\usepackage{amsmath,amssymb,amsthm,graphicx,color}
\usepackage{diagrams}
\usepackage{hyperref} 
\hypersetup{colorlinks=true} 
\newcommand{\source}{\mathtt{s}}

\newtheorem{thm}{Theorem}
\newtheorem{lemma}[thm]{Lemma}

\newtheorem{cor}[thm]{Corollary}

\theoremstyle{definition}
\newtheorem*{definition}{Definition}

\let\oldmarginpar\marginpar
\renewcommand\marginpar[1]{\-\oldmarginpar[\raggedleft\footnotesize #1]%
{\raggedright\footnotesize #1}}

\title[Local-to-Global Principles]{Local-to-Global Principles for Rotor Walk}
\author[Giacaglia, Levine, Propp, Zayas-Palmer]{Giuliano Pezzolo Giacaglia, 
Lionel Levine, \\ James Propp and Linda Zayas-Palmer}
\date{August 8, 2011}

\address{MIT Department of Mathematics, 77 Massachusetts Ave., Cambridge, MA 02139. {\tt giu} at {\tt mit} dot {\tt edu}}
\address{MIT Department of Mathematics, 77 Massachusetts Ave., Cambridge, MA 02139. \url{http://math.mit.edu/~levine}}
\address{Department of Mathematical Sciences, University of Massachusetts, Lowell. \url{http://faculty.uml.edu/jpropp/}}
\address{MIT Department of Mathematics, 77 Massachusetts Ave., Cambridge, MA 02139. {\tt lzayasp} at {\tt mit} dot {\tt edu}}



\thanks{This work was supported by MIT's Undergraduate Research 
Opportunities Program, NSF grants 0644877 and 1001905, an NSF postdoctoral research fellowship, and NSF-REU funding provided through the courtesy 
of Prof.\ Richard Stanley.}

\newcommand{\Z}{\mathbb{Z}}
\newcommand{\N}{\mathbb{N}}

\newcommand{\tfaea}{{\rm (a)}}
\newcommand{\tfaeb}{{\rm (b)}}
\newcommand{\tfaec}{{\rm (c)}}

\newcommand{\old}[1]{}
\newcommand{\arxiv}[1]{{\tt \href{http://arxiv.org/abs/#1}{arXiv:#1}}}
\newcommand{\mailto}[1]{\href{mailto:#1}{\nolinkurl{#1}}}

\begin{document}

\begin{abstract}
\noindent
In rotor walk on a finite directed graph, 
the exits from each vertex follow a prescribed periodic sequence.
Here we consider the case of rotor walk 
where a particle starts from a designated source vertex 
and continues until it hits a designated target set, 
at which point the walk is restarted from the source.
We show that the sequence of successively hit targets, 
which is easily seen to be \emph{eventually} periodic, is in fact periodic.
We show moreover that reversing the periodic patterns of all rotor sequences 
causes the periodic pattern of the hitting sequence to be reversed as well.
The proofs involve a new notion of equivalence of rotor configurations, 
and an extension of rotor walk incorporating time-reversed particles.
\end{abstract}

\maketitle

\section{Introduction} \label{sec-intro}

A rotor walk in a graph $G$ is a walk in which the sequence of exits
from each vertex is periodic.  The sequence of exits from a vertex $v$
is called the \emph{rotor mechanism} at $v$.  Rotor walks have been
studied in combinatorics as deterministic analogues of random walks,
in computer science as a means of load-balancing and territory
exploration, and in statistical physics as a model of self-organized
criticality.  In this paper, we explore several properties of the
rotor mechanism that imply corresponding properties of the {\it hitting 
sequence} when $G$ comes with a set of designated {\it target vertices}:

\begin{itemize}
\item Given a (periodic) rotor mechanism at each vertex,
the hitting sequence of the associated rotor walk
is periodic (Theorem~\ref{thm-periodic}).
\item If every rotor mechanism is palindromic, then the hitting sequence is
palindromic (Theorem~\ref{thm-palindromic}).
\item If every rotor mechanism is $m$-repetitive, then the hitting sequence
is $m$-repetitive (Theorem~\ref{thm-repeat}).
\end{itemize}

See below for precise definitions.  Since the rotor mechanisms are
local features of the walk --- each one depends only on the exits from a
particular vertex --- while the hitting sequence is a global feature, 
we regard these theorems as \emph{local-global principles}.

Let $G = (V,E)$ be a strongly connected finite directed graph, 
with self-loops and multiple edges permitted.
For a vertex $v \in V$, let $d(v)$ denote the outdegree of $v$.  
A \emph{rotor mechanism} at $v$ is an ordering of the directed edges 
(or ``arcs'') emanating from $v$, say as $e_v^{i}$ for $1 \leq i \leq d(v)$.
Let $v^{i}$ denote the endpoint of the arc $e_v^{i}$.
We extend the definition of $e_v^{i}$ and $v^{i}$ to all $i \in \Z$
by taking $e_v^{i}$ and $v^{i}$ to be periodic in $i$ with period $d(v)$.  
We often indicate the rotor mechanism at $v$ using the notation
\[ v \rightarrow v^{1},v^{2},\dots,v^{d(v)},\dots 
\ \mbox{(period $d(v)$)}. \]
Given a rotor mechanism at each vertex $v$, a {\it rotor walk\/} 
on $G$ is a finite or infinite sequence of vertices $x_0,x_1,x_2,\dots$ 
in which the $i$-th occurrence of $v$ is followed immediately 
by an occurrence of $v^{i}$.
For example, if the vertex set of $G$ is $\{1,2,3,4,5\}$
and the rotor mechanisms are
\begin{eqnarray*}
1 & \rightarrow & 3,4,5,\dots \ \mbox{(period 3)} \\
2 & \rightarrow & 3,\dots \ \mbox{(period 1)} \\
3 & \rightarrow & 4,2,\dots \ \mbox{(period 2)} \\
4 & \rightarrow & 1,\dots \ \mbox{(period 1)} \\
5 & \rightarrow & 1,\dots \ \mbox{(period 1)}
\end{eqnarray*}
then the rotor walk starting from 1 is
\[ 1,3,{\bf 4},1,{\bf 4},1,{\bf 5},
1,3,2,3,{\bf 4},1,{\bf 4},1,{\bf 5},
1,3,2,3,{\bf 4},1,\dots \]
which is eventually periodic with period 9.  Note that this sequence is not 
itself periodic (the initial $1$ does not repeat) but if we isolate the terms 
equal to $4$ or $5$ we obtain the sequence 
	\[ 4,4,5,4,4,5,\ldots \]
which is periodic with period 3.

In general, we assume that $G$ 
comes with a designated {\it source vertex\/} $s$
that serves as the starting point of the rotor walk
and a non-empty set $T$ of designated {\it target vertices\/}
such that all arcs emanating from a target vertex $t \in T$ go to $s$.  
In the example above, $s=1$ and $T=\{4,5\}$.  
A rotor walk starts at the source vertex 
and always returns to the source vertex
immediately after visiting a target vertex.  
We define the {\it hitting sequence\/}
as the subsequence of the rotor walk 
consisting of the terms that belong to $T$.  
In the above example, the hitting sequence is $4,4,5,4,4,5,\ldots$.

It is easy to show that the hitting sequence is infinite
(that is, the set of targets is visited infinitely often);
see Lemma~\ref{prop-infinite}, below.
As explained in \textsection\ref{sec-prelim}, it is also easy to show that
the hitting sequence is eventually periodic. 
Our first main result goes further:

\begin{thm} \label{thm-periodic}
The hitting sequence determined by a (periodic) rotor mechanism is periodic.
\end{thm}

As we have already seen, 
the rotor walk itself is typically \emph{not} periodic.
Let $\gamma_n$ be the portion of the walk 
strictly between the $n$-th and $(n+1)$-st visits to $T$.  
In the preceding example, the sequence $\{\gamma_n\}_{n \geq 1}$ is
	\[ 13, 1, 1, 1323, 1, 1, 1323, 1, 1, \ldots. \]
In general, this sequence is eventually periodic but is not periodic.

A natural question is how to determine the period of the hitting sequence.  
We will see that this period divides the order of a certain element 
of the \emph{sandpile group} $S(G/T)$ of the graph $G$ 
with the target set $T$ collapsed to a single vertex 
(Lemma~\ref{prop-theperiod}).

Our second main result states that 
if we {\it reverse\/} the rotor mechanism at each vertex by replacing
$$v \rightarrow v^{1},v^{2},\dots,v^{d(v)},\dots 
\ \mbox{(period $d(v)$)}$$
by
$$v \rightarrow v^{d(v)},v^{d(v)-1},\dots,v^{1},\dots 
\ \mbox{(period $d(v)$)}$$
for each vertex $v$,
then the hitting sequence undergoes an analogous reversal;
specifically, if the original hitting sequence has period $D$,
then the new hitting sequence will also have period $D$,
and for all $1 \leq i \leq D$ the $i$-th term of the new hitting sequence
will equal the $(D+1-i)$-th term of the original hitting sequence.
That is:

\begin{thm} \label{thm-reversal}
Reversing the periodic pattern of all rotor mechanisms
results in reversing the periodic pattern of the hitting sequence.
\end{thm}

E.g., for the above example, the reversed rotor mechanism
\begin{eqnarray*}
1 & \rightarrow & 5,4,3,\dots \ \mbox{(period 3)} \\
2 & \rightarrow & 3,\dots \ \mbox{(period 1)} \\
3 & \rightarrow & 2,4,\dots \ \mbox{(period 2)} \\
4 & \rightarrow & 1,\dots \ \mbox{(period 1)} \\
5 & \rightarrow & 1,\dots \ \mbox{(period 1)}
\end{eqnarray*}
gives the reversed hitting sequence $5,4,4,5,4,4,\dots$.

An immediate corollary of Theorem~\ref{thm-reversal}
is that if the rotors are all {\it palindromic\/}
(that is, if each fundamental period of each rotor
reads the same backwards and forwards)
then the same is true of the hitting sequence.

\begin{thm} \label{thm-palindromic}
If all rotor mechanisms are palindromic,
then the hitting sequence is palindromic.
\end{thm}

One can think of the entire collection of rotor mechanisms on $G$ 
as a single rotor --- 
perhaps embedded as a component of a larger system --- 
whose rotor mechanism is the hitting sequence.
From this perspective, 
Theorems~\ref{thm-periodic} and~\ref{thm-palindromic}
are local-to-global principles
asserting that if the sequence of exits from each vertex
possesses a certain property (periodicity, palindromicity), then
the hitting sequence has the same property.
We now state one more result of this type, Theorem~\ref{thm-repeat}.  
Further examples of local-global principles 
include Lemma~\ref{prop-infinite}, below, and \cite[Theorem~1]{HP10}.

Call a sequence $\{u_i\}_{i\geq 1}$ \emph{$m$-repetitive} 
if it consists of blocks of $m$ consecutive equal terms; that is,
	\[ u_{am+1} = u_{am+2} = \ldots = u_{am+m} \]
for all $a \geq 0$.  

\begin{thm} \label{thm-repeat}
If all rotor mechanisms are $m$-repetitive, 
then the hitting sequence is $m$-repetitive.
\end{thm}

For example, consider the $2$-repetitive rotor mechanism
\begin{eqnarray*}
1 & \rightarrow & 3,3,2,2,\dots \ \mbox{(period 4)} \\
2 & \rightarrow & 1,1,4,4,\dots \ \mbox{(period 4)} \\
3 & \rightarrow & 1,\dots \ \mbox{(period 1)} \\
4 & \rightarrow & 1,\dots \ \mbox{(period 1)}
\end{eqnarray*}
with source 1 and targets 3 and 4.  
The sequence of paths $\gamma_n$ taken by the walker until it hits a target
	\[ 13, 13, 121213, 13, 124, 124, \dots \ \mbox{(period 6)}  \]
is not $2$-repetitive, but the hitting sequence
	\[ 3,3,3,3,4,4,  \dots \ \mbox{(period 6)} \]
is $2$-repetitive.

The proof of Theorem~\ref{thm-repeat} is not difficult 
(see \textsection\ref{sec-prelim})
and uses only the abelian property of rotor walk (Lemma~\ref{prop-abelian}).
The proofs of Theorems~\ref{thm-periodic}-\ref{thm-palindromic} 
make essential use of a new notion of equivalence of rotor configurations.  
We summarize the highlights here, 
referring the reader to \textsection\ref{subsec-equiv} for the full definitions.

Let $V_0 = V - T$.
A \emph{rotor configuration} is a map $\rho : V_0 \to E$ such that 
$\rho(v)$ is an arc emanating from $v$; 
the arc $\rho(v)$ represents the arc 
by way of which a particle most recently exited vertex $v$.  
A \emph{particle configuration} is a map $\sigma : V_0 \to \N$; 
we interpret $\sigma(v)$ as the number of particles present at vertex $v$.  
Following~\cite{HLMPPW} we define 
an action $(\sigma, \rho) \mapsto \sigma \rho$ 
of particle configurations on rotor configurations.  
We then define rotor configurations $\rho_1$ and $\rho_2$ to be equivalent, 
written $\rho_1 \equiv \rho_2$, 
if there exists a particle configuration $\sigma$ such that 
$\sigma \rho_1 = \sigma \rho_2$
(Lemma~\ref{prop-equiv-tfae} will show that this is an equivalence relation.
In fact, $\rho_1 \equiv \rho_2$ if and only if
$\sigma \rho_1 = \sigma \rho_2$ for all ``sufficiently large'' $\sigma$,
in a sense made precise by part (d) of Lemma~\ref{prop-equiv-tfae}.)
We define an operation called \emph{complete cycle pushing} 
which takes an arbitrary rotor configuration $\rho$ as input 
and produces an acyclic rotor configuration $\rho^\dagger$ as output.

\begin{thm} ~
\begin{enumerate}
\item[(i)]
Each equivalence class of rotor configurations 
contains a unique acyclic configuration.  
\item[(ii)] The unique acyclic configuration equivalent to $\rho$ is $e\rho$, 
where $e$ is the recurrent identity element of the sandpile group.  
\item[(iii)] $e\rho = \rho^\dagger$ is the result of 
performing complete cycle pushing on $\rho$.
\end{enumerate}
\end{thm}

To see the relevance of this notion of equivalence 
to Theorem~\ref{thm-periodic}, 
let $\rho_n$ be the rotor configuration 
immediately after the rotor walk hits the target set $T$ for the $n$-th time.  
The sequence $\{\rho_n\}_{n \geq 0}$ is not periodic, but we will show that 
the sequence of equivalence classes $[\rho_0], [\rho_1], [\rho_2], \dots$ 
is periodic.  We then show that which target is hit by rotor walk 
starting at $s$ with rotor configuration $\rho$ 
depends only on the equivalence class $[\rho]$.

In the proof of Theorem~\ref{thm-reversal}, 
a helpful trick is the use of {\it antiparticles}
that behave like the ``holes'' considered in \cite{FL10}: 
while a particle at vertex $v$ first increments (progresses) the rotor at $v$
and then moves to a neighbor according to the updated rotor, 
an antiparticle at $v$ first moves to a neighbor 
according to the \emph{current} rotor at $v$
and then decrements (regresses) the rotor at $v$.
Reversing the rotor mechanism at each vertex is equivalent 
to replacing all particles by antiparticles and vice versa.

\subsection*{Related Work}
Rotor walk was first studied in computer science from the point of view 
of autonomous agents patrolling a territory \cite{WLB96}, 
and in statistical physics 
as a model of self-organized criticality \cite{PDDK96}.  
It is an example of a ``convergent game'' 
of the type studied by Eriksson \cite{Eri96} 
and more generally of an \emph{abelian network} 
of communicating finite automata.  
Abelian networks were proposed by Dhar \cite{Dha06}, 
and their theory is developed in \cite{BL11}. 

Rotor walk on $G$ reflects certain features of random walk on $G$ \cite{CS06}.
For vertices $v,w \in V$ let $d(v,w)$ be the number of arcs from $v$ to $w$, 
and consider the Markov chain on state space $V$
in which the transition probability from $v$ to $w$
equals $d(v,w)/d(v)$.
The frequency $p_i$ with which a particular target vertex $t$
occurs in the hitting sequence for rotor walk
equals the probability that the Markov chain 
when started from the source $s$ reaches $t$ 
before it reaches any other target vertex.
A main theme of \cite{HP10} and the companion article \cite{Pro11} is that
the ``global'' discrepancy between $n p_i$ and the number of times 
the rotor walk hits the target $t$ in the first $n$ runs is bounded --- 
independently of $n$ --- by a sum of ``local discrepancies'' 
associated with the rotors.

A special case of the periodicity phenomenon was noted by Angel and
Holroyd.  If $G$ is the $b$-regular tree of height $h$ and $T$ is the
set of leaves, it follows from the proof of Theorem 1.1 of \cite{LL09}
that the hitting sequence from the root is eventually periodic with
period $\# T$, and its fundamental period is a permutation of $T$.  In
Proposition 22 of \cite{AH11} Angel and Holroyd prove that for any
initial setting of the rotors, the first $\# T$ terms of the hitting
sequence are in fact a permutation of $T$.


\section{Abelian property, monoid action and group action} \label{sec-prelim}

This section collects the results from the literature that we will use.  All of these can be found in the survey \cite{HLMPPW}, and many date from considerably earlier; where we know of an earlier reference, we indicate that as well.
Let $d(v,w)$ be the number of arcs 
from $v$ to $w$ in the finite directed graph $G$.
Let $d(v) = \sum_{w \in V} d(v,w)$ be the out-degree of $v$.
Let $s$ be a designated source vertex
and $T$ be a set of designated target vertices.  
Arcs emanating from target vertices play no role in our argument; 
however, it can be helpful to imagine that for every target $t \in T$ 
we have $d(t) = d(t,s) = 1$ 
(that is, each target has just one outgoing arc, 
which points back to the source).  
We define $V_0 = V - T$, the set of non-target vertices.
We allow vertices in $V_0$ to have arcs pointing to $s$.

In \textsection\ref{sec-intro},
we defined a rotor walk on $G$ as 
an infinite sequence of vertices $x_0,x_1,x_2,\dots$
in which the $i$-th occurrence of $v$
is followed immediately by an occurrence of $v^{i}$.  
The proofs make use of an alternative, ``stack-based'' 
picture of rotor walk, which we now describe.  
This viewpoint goes back at least to \cite{DF91,Wil96}.

At each vertex $v$ is a bi-infinite stack of cards, 
in which each card is labeled by an arc of $G$ emanating from $v$.
The $i$-th card is labeled by the arc $e_i = (v,v^{i})$.  
For $i>0$, the $i$-th card in the stack represents an instruction for 
where the particle should step upon visiting vertex $v$ for the $i$-th time.  
(When $i<0$, the $i$-th card in the stack never gets used,
but it is helpful to pretend that it was used in the past
before the rotor walk began;
this point of view will play an important role
in the proof of Theorem~\ref{thm-palindromic}.)
We also have a pointer at $v$ that keeps track of
how many departures from $v$ have already occurred;
this pointer moves as time passes.
When $i$ departures from $v$ have occurred during the rotor walk thus far,
we represent the state of the stack and pointer as
\[ [ \dots, e_{i-2}, e_{i-1}, e_{i} | e_{i+1}, e_{i+2}, e_{i+3}, \dots ] \]
($i$ is 0 at the start of the rotor walk).
Arcs $e_{j}$ with $j \leq i$ to the left of the pointer
constitute the ``past'' of the rotor 
(arcs previously traversed), 
while the $e_{j}$'s with $j > i$
constitute the ``future'' of the rotor 
(arcs to be traversed on future visits to $v$).
The arc $e_{i}$ is called the \emph{retrospective state} of the rotor.  
It represents the most recent arc traversed from $v$.  
Arc $e_{i+1}$ is called the \emph{prospective state} of the rotor.  
It represents the next arc to be traversed from $v$.
When the particle next exits $v$ (along arc $e_{i+1}$)
the pointer moves to the right and the stack at $v$ becomes
\[ [ \dots, e_{i-2}, e_{i-1}, e_{i}, e_{i+1} | e_{i+2}, e_{i+3}, \dots ] . \]

The defining property of rotor walk is that for each vertex $v$, 
the sequence of labels in its stack is periodic.
Initially, however, we will not need this assumption.  
We use the term \emph{stack walk} to describe the more general situation 
when the stack at each vertex $v$ may be 
an arbitrary sequence of arcs emanating from $v$.

We will assume throughout that the finite directed graph $G$ 
is \emph{strongly connected}; 
that is, for any pair of vertices $v$ and $w$ there exist 
directed paths from $v$ to $w$ and from $w$ to $v$.  
Note that strong connectivity is a global property of $G$.  In fact, it is the only non-local ingredient needed for our local-global principles.  
The next lemma provides a simple example of 
how strong connectivity parlays a local property --- 
one that can be checked for each stack individually --- 
into a corresponding global property of the hitting sequence.

A sequence $a_1,a_2, \ldots$ whose terms belong to an alphabet $A$ 
is called \emph{infinitive} if for every $a \in A$ 
there are infinitely many indices $i$ such that $a_i=a$ \cite{Kim97}.  
Thus, we say that the stack at vertex $v$ is infinitive 
if every outgoing arc from $v$ appears infinitely often as a label.  
Likewise, the hitting sequence is infinitive 
if the walk hits every target $t$ infinitely often.

\begin{lemma} \label{prop-infinite}
If all stacks are infinitive, then the hitting sequence is infinitive.
\end{lemma} 

\begin{proof}
Since $G$ is finite, the stack walk visits at least one vertex infinitely often.  
If the walk visits~$v$ infinitely often, 
then since the stack at~$v$ is infinitive, 
the walk traverses every outgoing arc from~$v$ infinitely often, 
so it visits all of the out-neighbors of~$v$ infinitely often.  
Since $G$ is strongly connected, 
every vertex is reachable by a directed path of arcs from~$v$, 
so every vertex is visited infinitely often.  
In particular, the walk hits every target infinitely often.
\end{proof}

\subsection{Abelian property}

Suppose that several indistinguishable particles 
are present on vertices of $G$.
At each moment, one has a choice of which particle to move; 
one chooses a particle, 
shifts the pointer in the stack at the corresponding vertex, 
and advances that particle to a neighboring vertex 
according to the instruction on the card that the pointer just passed.  
We call this procedure a \emph{firing}.

For example, if we begin with $m$ particles at the source vertex~$s$,
we can repeatedly advance one of them until it hits a target,
then repeatedly advance another particle until it too hits a target,
and so on, until all the particles have hit (and remain at) targets.

The following lemma is known as the \emph{abelian property} of rotor-routing 
(another name for it is the ``strong convergence property,'' 
following Eriksson \cite{Eri96}).  
For a proof, see \cite[Theorem 4.1]{DF91} or \cite[Lemma 3.9]{HLMPPW}.

\begin{lemma} \label{prop-abelian}
Starting from particle configuration $\sigma$ and rotor configuration $\rho$, 
let $v_1, \ldots, v_m$ be a sequence of firings that results 
in all particles reaching the target set.
Let $N(t)$ be the number of particles that hit target $t$.
The numbers $N(t)$ ($t \in T$) and the final rotor configuration 
depend only on $\sigma$ and $\rho$; 
in particular, they do not depend on the sequence $v_1, \ldots, v_m$.
\end{lemma}

The abelian property is all that is needed to prove Theorem~\ref{thm-repeat}, 
which says that if every rotor mechanism is $m$-repetitive, 
then the hitting sequence is $m$-repetitive.

\begin{proof}[Proof of Theorem~\ref{thm-repeat}]
It suffices to show for all $n$ that if we feed $mn$ particles
through the system in succession 
(starting them at $s$ and stopping them when they hit $T$), 
then the number of particles that hit each target is a multiple of $m$;
for, if we know this fact for both $mn$ and $m(n+1)$, 
then it follows that the $(mn+1)$-st through $(mn+m)$-th particles 
must all hit the same target.
By the abelian property (Lemma~\ref{prop-abelian}),
if we let the $mn$ particles walk in tandem,
letting each particle take its $i$-th step
before any particle takes its $(i+1)$-st step, 
then since each stack is $m$-repetitive, 
the particles travel in groups of size $m$,
such that the particles in each group 
travel the same path and hit the same target.
\end{proof}

Note that Theorem~\ref{thm-repeat} did not require the stacks to be periodic.  
Theorems~\ref{thm-periodic}--\ref{thm-palindromic} certainly do require 
periodic stacks, so we make this assumption for the remainder of the paper.  

\subsection{Action of particle configurations on rotor configurations} \label{subsec-action}

Denote by $Q$ the set of particle configurations
        \[ Q = \{\sigma : V_0 \to \N \} \]
and by $R$ the set of rotor configurations
        \[ R = \{\rho : V_0 \to E \mid \source(\rho(v)) = v 
	\mbox{ for all $v \in V_0$} \} \]
where $\source(e)$ denotes the source of the arc $e$.
We give $Q$ the structure of a commutative monoid under pointwise addition.

Next we recall from \cite{HLMPPW} the construction of the action
	\[ Q \times R \to R. \]
Associated to each vertex $v \in V_0$ is a \emph{particle addition operator} $E_v$
acting on the set of rotor configurations: given a rotor configuration $\rho$, 
we define $E_v(\rho)$ as the rotor configuration obtained from $\rho$
by adding a particle at $v$ and letting it perform rotor walk 
until it arrives at a target.
Lemma~\ref{prop-abelian} implies that
the operators $E_v$ commute: $E_v E_w = E_w E_v$ for all $v,w \in V_0$.

Now given a particle configuration $\sigma$ on $G$, we define 
	\[ E_{\sigma} = \prod_{v \in V_0} (E_v)^{\sigma(v)} \] 
where the product denotes composition.  
Since the operators $E_v$ commute, the order of composition is immaterial.  
The action of particle configuration $\sigma$ on rotor configuration $\rho$ 
is defined by $\sigma \rho := E_\sigma(\rho)$.  
In words, $\sigma \rho$ is the rotor configuration obtained from $\rho$ 
by placing $\sigma(v)$ particles at each vertex $v$ 
and letting all particles perform rotor walk until they hit the target set $T$.  
By Lemma~\ref{prop-abelian}, the order in which the walks are performed 
has no effect on the outcome.  
The fact that the operators $E_v$ commute ensures 
that we have a well-defined action, that is, 
$(\sigma_1 + \sigma_2)\rho = \sigma_1 (\sigma_2 \rho)$.

\subsection{The sandpile monoid and its action on rotor configurations} \label{subsec-monoid}

A particle configuration $\sigma$ is called \emph{stable} if
	\[ \sigma(v) \leq d(v)-1 \qquad \mbox{for all \ $v \in V_0$.} \]
If $\sigma$ is not stable, we can \emph{stabilize} it 
by repeatedly {\em toppling} unstable vertices: Set $\sigma_0 = \sigma$, 
choose a vertex $v_0 \in V_0$ such that $\sigma_0(v_0) \geq d(v_0)$ 
and topple it by sending one particle along each outgoing arc from $v_0$.  
The resulting configuration $\sigma_1$ is given by
$$
\sigma_1(w) = \left\{ \begin{array}{ll}
\sigma_0(w) + d(v_0,w) & \mbox{if $w \neq v_0$}, \\
\sigma_0(w) + d(v_0,w) - d(v_0) & \mbox{if $w = v_0$}.
\end{array} \right.
$$
If $\sigma_1$ is not stable, 
choose a vertex $v_1$ such that $\sigma_1(v_1) \geq d(v_1)$ 
and topple it in the same way to arrive at a new configuration $\sigma_2$.  
Strong connectedness ensures that after finitely many topplings 
we reach a stable configuration, 
which is called the \emph{stabilization} of $\sigma$ 
and denoted $\sigma^\circ$.  
The stabilization $\sigma^\circ$ 
does not depend on the order of topplings~\cite{Dha90}.

Let $Q^\circ$ be the set of stable particle configurations.
We give $Q^\circ$ the structure of a commutative monoid with the operation
	\[ (\sigma_1, \sigma_2) \mapsto (\sigma_1 + \sigma_2)^\circ. \]
That is, we sum the configurations pointwise, and then stabilize.  
By comparing two different toppling orders 
to stabilize $\sigma_1+\sigma_2+\sigma_3$, we see that
$ ((\sigma_1 + \sigma_2)^\circ + \sigma_3)^\circ 
= (\sigma_1 + (\sigma_2 + \sigma_3)^\circ)^\circ$, 
which shows that this operation is associative.  
The monoid $Q^\circ$ is called the \emph{sandpile monoid} of $G$; 
its structure has been investigated in~\cite{BT10,CGGMS11}.

\begin{lemma}
\label{prop-monoidaction}
For any particle configuration $\sigma$ and any rotor configuration $\rho$ 
we have \[ \sigma^\circ \rho = \sigma \rho. \]  
Hence, the action of particle configurations on rotor configurations 
descends to an action of the sandpile monoid
	\[ Q^\circ \times R \to R. \]
\end{lemma}
%

\begin{proof}
\cite[Lemma 3.12]{HLMPPW}
We compute $\sigma \rho$ by grouping the initial rotor moves into ``batches''
each consisting of $d(v)$ moves from a vertex $v$.  
The net effect of a batch of rotor moves is the same 
as that of a toppling at~$v$: 
the rotor at $v$ makes a full turn, 
so the rotor configuration is unchanged, 
and one particle is sent along each arc emanating from $v$.  
After finitely many batches, we arrive at particle configuration $\sigma^\circ$ 
with rotors still configured as $\rho$.  
Now letting each remaining particle perform rotor walk 
until reaching the target set 
yields the rotor configuration $\sigma^\circ \rho$.
\end{proof}

We may express Lemma~\ref{prop-monoidaction} as a commutative diagram
\begin{diagram}
 Q \times R & \rTo & Q^\circ \times R \\
 & \rdTo & \dTo \\
 & & R
\end{diagram}
where the top arrow is $(\sigma,\rho) \mapsto (\sigma^\circ,\rho)$.

\subsection{The sandpile group and its action on spanning forests} \label{subsec-group}

We say that a stable particle configuration $\tau \in Q^\circ$ 
is \emph{reachable} from a particle configuration $\sigma$ if there exists 
a particle configuration $\tau'$ such that $\tau = (\tau' + \sigma)^\circ$.
We say that $\tau$ is \emph{recurrent} 
if it is reachable from any $\sigma \in Q$.

Note that if $\tau$ is recurrent, 
then for any $\sigma \in Q$ there exists $\tau_1$ such that 
$\tau_1(v) \geq \sigma(v)$ for all $v$ and $\tau_1^\circ = \tau$;
indeed, since $\tau$ is reachable from $\sigma$, 
there exists $\tau' \in Q$ with $\tau = (\tau'+ \sigma)^\circ$, 
and we can take $\tau_1 = \tau' + \sigma$.

Denote by $S(G/T)$ the set of recurrent particle configurations.
If $\tau$ is recurrent and $\sigma$ is any particle configuration, 
then $(\sigma + \tau)^\circ$ is recurrent. 
That is, the set $S(G/T)$ 
is an \emph{ideal} of the monoid $Q^\circ$.  
In fact, $S(G/T)$ is the minimal ideal of $Q^\circ$, 
which shows that it is an abelian group \cite{BT10}.  
This group is called the \emph{sandpile group} of $G$ 
relative to the target set $T$. 
The set $T$ plays the role of the sink vertex in \cite{HLMPPW}.  
In the terminology of that paper, 
$S(G/T)$ is the sandpile group of the graph $G/T$ 
obtained by collapsing $T$ to a single vertex.

A rotor configuration $\rho$ is \emph{acyclic} if 
the graph $(V,\rho(V_0))$ contains no oriented cycles
(where $\rho(V_0) = \{\rho(v): \ v \in V_0\}$).
Equivalently, the rotors $\{\rho(v)\}_{v \in V_0}$ 
form an oriented spanning forest of $G$ rooted at $T$.

\begin{lemma} \cite{HLMPPW} \label{prop-permutation}
Each addition operator $E_{\sigma}$ acts as a permutation
on the set $R_0$ of acyclic rotor configurations.  
Thus the action of the sandpile monoid $Q^\circ$ on rotor configurations 
restricts to an action
	\[ S(G/T) \times R_0 \to R_0 \]
of the sandpile group $S(G/T)$ on acyclic rotor configurations.  
\end{lemma}

A further result proved in \cite{HLMPPW} is that 
this group action is free and transitive.  
In other words, if $\rho$ and $\rho'$ 
are two spanning forests of $G$ rooted at $T$, 
then there is a unique element of the sandpile group $\sigma \in S(G/T)$ 
such that $\sigma \rho = \rho'$.
In particular, the order $|S(G/T)|$ of the sandpile group 
equals the number of acyclic rotor configurations $|R_0|$, 
which is the number of spanning forests of $G$ rooted at $T$.  
We will not use these facts, however, 
except for a brief aside (Lemma~\ref{prop-theperiod}) 
where we identify the period of a sequence that arises 
in the proof of Theorem~\ref{thm-periodic}; 
there we use the freeness of the action.

The identity element $e \in S(G/T)$ is a highly nontrivial object 
(see for instance \cite[Figures 4--6]{HLMPPW})
and plays a role in several of our lemmas below.  
If $G$ has an oriented cycle, 
then $e$ is distinct from the identity element $\mathbf{0}$ of $Q^\circ$ 
because the latter is not recurrent.

\section{Equivalence and cycle pushing} \label{sec-lemmas} 

In this section we develop a new notion of equivalence of rotor configurations
and use it to prove Theorem~\ref{thm-periodic}.
Here and throughout the rest of the article,
we assume that the stack at each vertex $v$
is periodic with period $d(v)$.
For $e = e^{i}_v$
(the $i$th arc in the rotor mechanism at $v$)
we define $e^+ = e^{i+1}_v$ and $e^- = e^{i-1}_v$,
where $i+1$ and $i-1$ are to be interpreted modulo $d(v)$.

\subsection{Equivalence of rotor configurations} \label{subsec-equiv} 

\begin{lemma} \label{prop-equiv-tfae}
Let $\rho$ and $\rho'$ be rotor configurations on $G$.
The following are equivalent:
\begin{enumerate}
\item[\rm (a)]
$\sigma \rho = \sigma \rho'$ for some particle configuration $\sigma$.

\item[\rm (b)] $\sigma \rho = \sigma \rho'$ 
for all recurrent configurations $\sigma \in S(G/T)$.

\item[\rm (c)] $e\rho = e\rho'$, where $e$ is the identity element of $S(G/T)$.

\item[\rm (d)] $\sigma\rho = \sigma\rho'$ 
for all configurations $\sigma \geq e$.
\end{enumerate}
\end{lemma}

\begin{proof}
It suffices to show that (a) $\Rightarrow$ (b) and (c) $\Rightarrow$ (d),
since (b) $\Rightarrow$ (c) and (d) $\Rightarrow$ (a) trivially.

\noindent
(a) $\Rightarrow$ (b):
Suppose that $\sigma \rho = \sigma \rho'$ 
for some particle configuration $\sigma$, 
and let $\tau \in S(G/T)$ be a recurrent configuration.  
Since $\tau$ is recurrent, there exists a particle configuration $\tau_1$ 
such that $\tau_1^\circ = \tau$ and $\tau_1 \geq \sigma$.  
Writing $\tau_1 = \sigma + \sigma_1$ for some $\sigma_1 \geq 0$, we obtain
	\begin{align*} \tau \rho = \tau_1^\circ \rho 
	= \tau_1 \rho &= \sigma_1 (\sigma \rho) \\
	&= \sigma_1(\sigma \rho') = \tau_1 \rho' = \tau_1^\circ \rho' 
	= \tau \rho'. 
\end{align*}

\noindent
(c) $\Rightarrow$ (d):
If $\sigma \geq e$, then writing $\sigma = e + \tau$ we have
$\sigma \rho = \tau (e \rho) = \tau (e \rho') = \sigma \rho'$.
\qedhere 
\renewcommand{\qedsymbol}{}
\end{proof}

\begin{definition} 
Rotor configurations $\rho$ and $\rho'$ are \emph{equivalent}, 
denoted $\rho \equiv \rho$, 
if the four equivalent conditions of Lemma~\ref{prop-equiv-tfae} hold.
\end{definition}

From condition (c) of Lemma~\ref{prop-equiv-tfae} 
it is immediate that $\equiv$ is an equivalence relation.  
We write the equivalence class of $\rho$ as $[\rho]$.

\begin{lemma} \label{prop-respect}
If $\rho \equiv \rho'$ then $\tau \rho \equiv \tau \rho'$
for all particle configurations $\tau$.
\end{lemma}

\begin{proof}
If $\rho \equiv \rho'$,
then there exists $\sigma$ such that $\sigma \rho = \sigma \rho'$,
which implies $\sigma (\tau \rho) = \tau(\sigma \rho) = \tau(\sigma \rho') 
= \sigma (\tau \rho')$,
which implies $\tau \rho \equiv \tau \rho'$.
\end{proof}

We say that a rotor configuration $\rho$ is \emph{reachable} from $\rho'$ 
if there exists a particle configuration $\sigma \neq \mathbf{0}$ 
such that $\rho = \sigma \rho'$.
We say that a rotor configuration $\rho$ is \emph{recurrent} 
if it is reachable from itself.  
The following lemma encapsulates the remaining results of \cite{HLMPPW} 
that we will need.

\begin{lemma}
\label{prop-recur-tfae}
The following properties of a rotor configuration $\rho$ are equivalent:
\begin{enumerate}
\item[\tfaea] $\rho$ is recurrent.
\item[\tfaeb] $\rho$ is acyclic.
\item[\tfaec] $\rho = e\rho$.
\end{enumerate}
\end{lemma}

\begin{proof}
The equivalence of (a) and (b) is Lemma 3.15 of \cite{HLMPPW}.  
The implication (b)$\Rightarrow$(c) follows from the well-definedness 
of the action of $S(G/T)$ on acyclic rotor configurations 
(Lemma~\ref{prop-permutation}).
To see that (c)$\Rightarrow$(a), note that $e \rho = (ee) \rho = e(e\rho)$ 
is reachable from itself, hence recurrent.
\end{proof}

See \cite[Lemma~3.16]{HLMPPW} for several other conditions 
equivalent to being recurrent.

\begin{lemma} \label{prop-unique}
Each equivalence class of rotor configurations
contains exactly one that is recurrent.  
The unique recurrent configuration equivalent to $\rho$ is $e\rho$, 
where $e$ is the identity element of $S(G/T)$.
\end{lemma}

\begin{proof}
Let $\rho$ be any rotor configuration.  Then
	\[ e(e\rho) = (e^2)\rho = e \rho. \]
Hence $e \rho$ is recurrent by Lemma~\ref{prop-recur-tfae} 
and $e \rho \equiv \rho$.
So each equivalence class contains at least one recurrent configuration.

For the reverse direction, suppose that $\rho$ and $\rho'$ are both recurrent 
and that $\rho \equiv \rho'$.  
By Lemma~\ref{prop-recur-tfae} we have 
$\rho  = e \rho = e \rho' = \rho'$.
\end{proof}

As a consequence of Lemmas~\ref{prop-permutation},
\ref{prop-respect}, and~\ref{prop-unique}, we have

\begin{cor} \label{prop-permutation-too}
Each addition operator $E_{\sigma}$ acts as a permutation
on the set $R/\!\!\equiv$ of equivalence classes of rotor configurations.  
Thus the action of the sandpile monoid $Q^\circ$ on rotor configurations 
projects 
to an action
	\[ S(G/T) \times R/\!\!\equiv \ \to \ R/\!\!\equiv \]
of the sandpile group $S(G/T)$ on 
equivalence classes of rotor configurations.  
\end{cor}
 
\subsection{Cycle pushing} \label{subsec-push}
 
Lemma~\ref{prop-unique} gives one way to compute 
the unique recurrent rotor configuration equivalent to $\rho$:
first compute the identity element $e$ of the sandpile group of $G/T$, 
then add $e(v)$ particles at each vertex $v \in V_0$ and stabilize.  
Note however that this is rather inefficient;
for instance, in the case where $\rho$ is already acyclic,
a smart algorithm would recognize this fact
and simply output $\rho$ directly.
We now describe a more efficient way to compute $e \rho$.  
The idea is to convert $\rho$ into an acyclic configuration 
by successively removing cycles in the rotors.  
We call this process \emph{complete cycle pushing}.

\begin{figure}[htbp]
\centering
\includegraphics[height=.4\textheight]{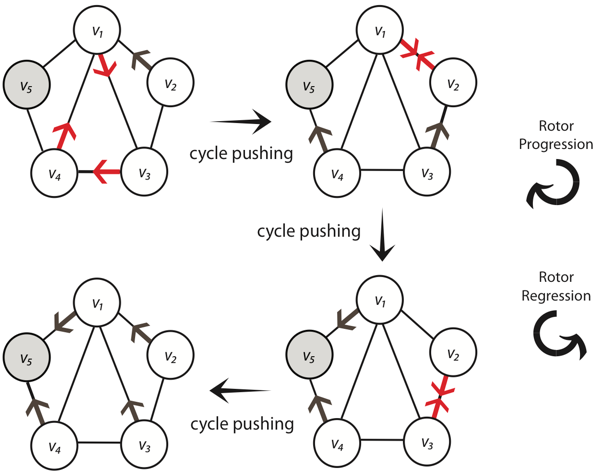}
\caption{Example of complete cycle pushing, 
starting from an arbitrary rotor configuration (upper left) 
to obtain an acyclic rotor configuration (lower left).
Here the target set is $T=\{v_5\}$.  
At each step, each rotor participating in a cycle (drawn in red) 
is regressed counterclockwise, until there are no more cycles.
}
\label{fig:cyclepushing}
\end{figure}

Cycle pushing is a key idea in Wilson's work on random stacks \cite{Wil96}
and in recent work on fast simulation of rotor-routing \cite{FL10}.
Suppose the rotor configuration $\rho$ contains a cycle
$\mathcal{C}$ with vertices $v_0,v_1,v_2,\dots,v_r=v_0$;
that is, for all $0 \leq j \leq r-1$,
the arc $\rho(v_j)$ points from $v_j$ to $v_{j+1}$.

The rotor configuration $\mathcal{C}\rho$ obtained 
by \emph{pushing} $\mathcal{C}$ is given by
	\[ \mathcal{C} \rho (v) = \begin{cases} 
	\rho(v)^- & \mbox{ if } v \in \mathcal{C}, \\
	\rho(v) & \mbox{ otherwise.} \end{cases} \]
In other words, for each $j=0,1,\ldots,r-1$ the rotor $\rho(v_j)$ is regressed,
and the other rotors remain unchanged.

Suppose we have a sequence of rotor configurations
$\rho_0, \rho_1, \dots, \rho_m$ where for each $i < m$ the configuration
$\rho_{i+1}$ is obtained from $\rho_{i}$
by pushing a cycle $\mathcal{C}_i$ in $\rho_{i}$,
and suppose moreover that $\rho_m$ is acyclic.
We say that $\rho_m$ is obtained from $\rho_0$ by {\em complete cycle pushing}.

Complete cycle pushing involves a choice of ordering 
in which to push the cycles $\mathcal{C}_i$.  
Wilson~\cite{Wil96} showed that these choices do not affect the outcome: 
if $\rho, \rho'$ are acyclic configurations 
that can be obtained from $\rho_0$ by complete cycle pushing, 
then $\rho = \rho'$.  
We will not use the uniqueness in our proofs: 
in fact, Lemma~\ref{prop-final} below gives another proof of Wilson's result.

Figure~\ref{fig:cyclepushing}
shows how a rotor configuration
is affected by cycle pushing.
The shaded vertex $v_5$ is the target vertex,
and the cycles that are pushed
(first the 3-cycle $v_1 \rightarrow v_3 \rightarrow v_4 \rightarrow v_1$,
then the 2-cycle $v_1 \rightarrow v_2 \rightarrow v_1$,
and then the 2-cycle $v_2 \rightarrow v_3 \rightarrow v_2$)
are shown in red.

\begin{lemma} \label{prop-push}
If $\rho'$ is obtained from $\rho$ by cycle pushing,
then $\rho' \equiv \rho$.
\end{lemma}

\begin{proof}
Let $\rho' = \mathcal{C} \rho$, and let $\sigma = 1_\mathcal{C}$ 
be the particle configuration consisting of 
one particle at each vertex $v_j$ of the cycle.
We claim that $\sigma \rho' = \sigma \rho$.
Starting from $\rho'$, let each particle take a single step of rotor walk: 
for each $j$, the particle at $v_j$ moves to $v_{j+1}$ (taking indices mod $r$),
and the rotor at $v_j$ progresses to $\rho'(v_j)^+ = \rho(v_j)$.  
Since each vertex $v_j$ on the cycle sends one particle to $v_{j+1}$ 
and receives one particle from $v_{j-1}$, 
the resulting particle configuration is still $\sigma$; 
on the other hand, the rotor configuration has changed from $\rho'$ to $\rho$.
By the abelian property
we conclude that $\sigma \rho' = \sigma \rho$, 
and hence $\rho' \equiv \rho$.
\end{proof}

\begin{lemma} \label{prop-terminate}
For any initial rotor configuration, 
any sequence of cycle pushing moves 
yields an acyclic configuration in finitely many steps.
\end{lemma}

\begin{proof}
Recall that target vertices do not have rotors.
Hence if a vertex $w$ has an arc to a target vertex $t$,
then $w$ can participate in only a finite number of cycle pushing moves,
because at some point the rotor at $w$ would point to $t$,
and thereafter $w$ cannot belong to a pushable cycle.
Thereafter, each vertex $v$ that has an arc to $w$
can participate in only a finite number of cycle pushing moves,
because at some point the rotor at $v$ would point to $w$,
and thereafter $v$ cannot belong to a pushable cycle.
Continuing in this fashion,
and using the strong connectedness of $G$, we see that 
every vertex can participate in only finitely many cycle pushing moves.
\end{proof}

\begin{lemma} \label{prop-final}
Let $\rho$ be a rotor configuration.
Any sequence of cycle pushing moves that starts from $\rho$
must terminate with $e \rho$,
the unique acyclic rotor configuration equivalent to $\rho$.
\end{lemma}

\begin{proof}
By Lemma~\ref{prop-terminate}, 
any sequence of cycle pushing moves starting from $\rho$ 
must terminate in an acyclic configuration $\rho'$.  
By Lemma~\ref{prop-push} we have $\rho' \equiv \rho$.  
Since $\rho'$ is acyclic, $\rho'$ is recurrent by Lemma~\ref{prop-recur-tfae},
and hence $\rho' = e \rho$ by Lemma~\ref{prop-unique}.
\end{proof}

\begin{figure}[htbp]
\centering
\includegraphics[height=.4\textheight]{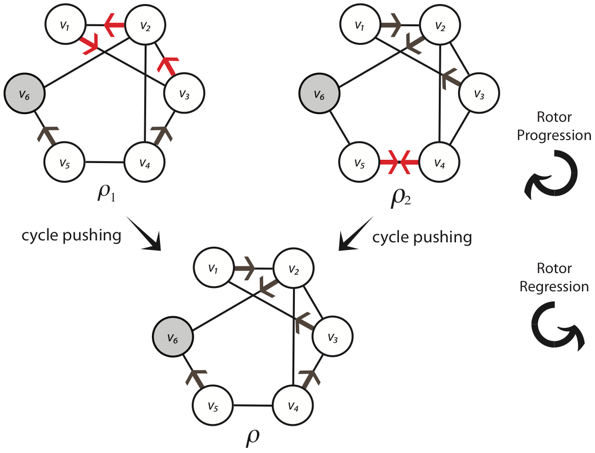}
\caption{The rotor configurations $\rho_1$ at top left and $\rho_2$ 
at top right yield the same acyclic configuration $\rho$ after 
complete cycle pushing, so they are equivalent by Lemma~\ref{prop-equivalent}.}
\label{fig:equivalentconfigs}
\end{figure}

The next lemma shows that equivalence of rotor configurations 
is the reflexive-symmetric-transitive closure of 
the relation $\rho \sim \mathcal{C} \rho$ given by cycle pushing.

\begin{lemma} \label{prop-equivalent}
$\rho_1 \equiv \rho_2$ if and only if there exists a rotor configuration that is
accessible from both $\rho_1$ and $\rho_2$
by a sequence of cycle pushing moves.
\end{lemma}

\begin{proof}
If $\rho_1 \equiv \rho_2$
then $e \rho_1 = e \rho_2$
is accessible from both $\rho_1$ and $\rho_2$ by Lemma~\ref{prop-final}.
Conversely, 
if $\rho'$ is a configuration accessible from both $\rho_1$ and $\rho_2$, 
then $\rho_1 \equiv \rho' \equiv \rho_2$ by Lemma~\ref{prop-push}.
\end{proof}

For a pictorial example, 
see Figure~\ref{fig:equivalentconfigs}.  The rotor configuration $\rho$ at the bottom is acyclic,
and the other two rotor configurations
lead to $\rho$ after a single cycle pushing move
(in one case, 
the 3-cycle $v_1 \rightarrow v_3 \rightarrow v_2 \rightarrow v_1$ is pushed, 
and in the other case, 
the 2-cycle $v_4 \rightarrow v_5 \rightarrow v_4$ is pushed).
As in Figure~\ref{fig:cyclepushing}, 
rotors progress by turning clockwise
and regress by turning counterclockwise.
Lemma~\ref{prop-equivalent} tells us
that the two non-acyclic rotor configurations $\rho_1$ and $\rho_2$
must be equivalent,
and indeed the reader can check that condition (a) 
of Lemma~\ref{prop-equiv-tfae} is satisfied
if one takes $\sigma$ to be the particle configuration
with a single particle at $v_4$;
that is, if we add a single particle at $v_4$ and let it perform rotor walk until reaching the target vertex $v_6$, then the two rotor configurations become the same.

Denote by $t_v(\rho)$ the target vertex reached by a particle started at $v$ if the initial rotor configuration is $\rho$.

\begin{lemma} \label{prop-same-target}
If $\rho_1 \equiv \rho_2$, then $t_v(\rho_1) = t_v(\rho_2)$ for all $v \in V_0$.
\end{lemma}

\begin{proof}
By Lemma~\ref{prop-equivalent} it suffices to consider the case where 
$\rho_2$ is obtained from $\rho_1$ by pushing
a cycle $v_0,v_1,\dots,v_r=v_0$.
If the particle added to $\rho_1$ at $v$ never hits the cycle,
then the particle added to $\rho_2$ at $v$ will traverse the exact same path,
arriving at the same target.
On the other hand, 
suppose the particle added to $\rho_1$ at $v$ hits the cycle, 
say at $v_0$.
Then the particle added to $\rho_2$ at $v$ will take the same walk to $v_0$
and will then traverse the cycle, arriving back at $v_0$.
At this point the rotor configuration will be the same
as the rotor configuration for the first situation 
(i.e., starting from $\rho_1$)
when the particle first hits $v_0$.
Thereafter, the two processes evolve identically,
since in both situations the particle is at $v_0$
and the rotor configurations at this stage
are the same in both evolutions.
In particular, the particle will end up at the same target vertex.
\end{proof}

\subsection{Proof of the periodicity theorem} \label{subsec-periodic}

We can now prove our first main result, that the hitting sequence 
associated with a (periodic) rotor mechanism is periodic.

\begin{proof}[Proof of Theorem~\ref{thm-periodic}] 
Let $t_1,t_2,\ldots$ be the hitting sequence 
for initial rotor configuration $\rho_0$, 
and for $n \geq 1$ let $\rho_n$ be 
the rotor configuration after $n$ particles released from the source vertex $s$
have hit the targets $t_1,\dots,t_n$ (staying put after each hit). 
Then $\rho_{n} = E_s \rho_{n-1}$ for $n \geq 1$.
Let $[\rho_n]$ denote the equivalence class of $\rho_n$.
Recall that $E_s$ acts as a permutation on equivalence classes
(Corollary~\ref{prop-permutation-too}),
so the sequence $[\rho_0], [\rho_1], [\rho_2], \dots$ is periodic,
say with period $D$.
Then by Lemma~\ref{prop-same-target}, since
$[\rho_{n+D}] = [\rho_n]$ for all $n \geq 0$, we conclude that $t_{n+D} = t_{n}$ for all $n \geq 1$, which shows that the hitting sequence is periodic with period~$D$.
\end{proof}

Next we identify the period $D$ of the sequence $[\rho_0], [\rho_1], \ldots$ 
arising in the proof of Theorem~\ref{thm-periodic}.  
This in turn gives an upper bound on the period 
of the hitting sequence $t_1,t_2,\ldots$, 
namely, the latter period is a divisor of $D$.  
Denote by $\delta_s$ the particle configuration 
consisting of $1$ particle at the source vertex, 
and let $g_s = (\delta_s + e)^\circ$ be 
the corresponding recurrent configuration.  

\begin{lemma}
\label{prop-theperiod}
Let $D$ be the order of $g_s$ in the sandpile group $S(G/T)$. 
The sequence of equivalence classes 
of rotor configurations $\{[\rho_n]\}_{n \geq 0}$ has period $D$.  
Moreover, the hitting sequence satisfies $t_{n+D} = t_n$ for all $n\geq 1$. 
\end{lemma}

\begin{proof}
For any rotor configuration $\rho$, 
since $\rho \equiv e\rho$ we have by Lemma~\ref{prop-respect}
	\[ g_s \rho = (\delta_s + e)^\circ \rho = (\delta_s + e)\rho = E_s(e\rho) \equiv E_s \rho. \]
Since $g_s^D = e$, we obtain 
	\[ E_s^D \rho \equiv g_s^D \rho = e \rho \equiv \rho \]
which shows that $\rho_{n+D} \equiv \rho_n$ for all $n\geq 0$.  
Conversely, if $\rho_{n+k} \equiv \rho_n$ for some $n\geq 0$ and $k \geq 1$, 
then $g_s^k \rho_n \equiv \rho_n$, which implies $g_s^k = e$ 
since the action of $S(G/T)$ on equivalence classes of rotor configurations 
is free; hence $k$ must be divisible by $D$.

The fact that $t_{n+D} = t_n$ for all $n \geq 1$ 
follows from Lemma~\ref{prop-same-target}.
\end{proof}

\section{Time reversal and antiparticles} \label{sec-reversal}

\subsection{Stack flipping} \label{subsec-flip}

Recall the stacks picture introduced in \textsection\ref{sec-prelim}.
Each vertex $v \in V_0$ has a stack $\rho_v$, 
which is a bi-infinite sequence of arcs
\[ \rho_v = [ \dots, e_{-2}, e_{-1}, e_{0} | e_{1}, e_{2}, e_{3}, \dots ]. \]
(We abuse notation slightly by using the same letter ($\rho$) 
to denote a stack configuration and its corresponding rotor configuration.)
The $e_{i}$ with $i \leq 0$ constitute the ``past'' of the stack,
the $e_{i}$ with $i > 0$ constitute the ``future'' of the stack,
$e_{0}$ is the retrospective state of the stack,
and $e_{1}$ is the prospective state of the stack;
the pointer ``$|$'' marks the divide between past and future.
When a particle at $v$ takes a step, the pointer shifts to the right,
so that the stack at $v$ becomes
\[ [ \dots, e_{-2}, e_{-1}, e_{0}, e_{1} | e_{2}, e_{3}, \dots ] \]
and the particle travels along arc $e_{1}$.

Shifting the pointer at $v$ to the right
corresponds to progressing the rotor at $v$, 
or in stack language, popping the stack at $v$;
correspondingly, shifting the pointer at $v$ to the left
will be called regressing the rotor or {\em pushing} the stack at $v$.
When we perform cycle pushing,
the pointer for the vertex $v$ moves one place to the left
for all vertices $v$ belonging to the cycle.


We define {\it stack flipping\/} as the operation on a bi-infinite stack
that exchanges past and future, turning
\[ [ \dots, e_{-2}, e_{-1}, e_{0} | e_{1}, e_{2}, e_{3}, \dots ] \]
into
\[ [ \dots, e_{3}, e_{2}, e_{1} | e_{0}, e_{-1}, e_{-2}, \dots ] . \]
Given a stack configuration $\rho = (\rho_v)_{v \in V_0}$, let $\Phi(\rho)$ denote the stack configuration obtained by flipping all its stacks.  Note that $\Phi(\Phi(\rho)) = \rho$.

\begin{lemma} \label{prop-flip-reverse}
Let $\rho$ be a rotor configuration that has a cycle $\mathcal{C}$.  
Then $\mathcal{C}$ is also a cycle of $\Phi(\mathcal{C} \rho)$, and 
	\[ \Phi(\mathcal{C} (\Phi(\mathcal{C}\rho))) = \rho. \]
\end{lemma}

\begin{proof}
Let $v$ be a vertex of $\mathcal{C}$.  Let $\rho' = \mathcal{C} \rho$, 
and write the rotor stack at $v$ as
\[ \rho_v = [\dots, e_{-2}, e_{-1}, e_{0} | e_{1}, e_{2}, e_{3}, \dots]. \]
If we push the cycle, the stack at $v$ becomes
\[ (\mathcal{C} \rho)_v = 
[\dots, e_{-2}, e_{-1} | e_{0}, e_{1}, e_{2}, e_{3}, \dots]. \]
If we then flip all the stacks, we obtain
\[ (\Phi(\mathcal{C} \rho))_v = 
[\dots, e_{3}, e_{2}, e_{1}, e_{0} | e_{-1}, e_{-2}, \dots]. \]
The retrospective rotors at the vertices $v \in \mathcal{C}$
are now as they were initially in $\rho$, 
so they form the same cycle $\mathcal{C}$.  Pushing that cycle yields
\[ (\mathcal{C} ( \Phi(\mathcal{C} \rho)))_v = 
[\dots, e_{3}, e_{2}, e_{1} | e_{0}, e_{-1}, e_{-2}, \dots]. \]
Finally, flipping the stacks once more brings us to
\[  (\Phi(\mathcal{C} ( \Phi(\mathcal{C} \rho))))_v = 
[\dots, e_{-2}, e_{-1}, e_{0} | e_{1}, e_{2}, e_{3}, \dots] \]
which equals $\rho_v$.

Meanwhile, for those vertices $v$ that are not part of the cycle $\mathcal{C}$,
the stack at $v$ is simply reversed twice
(with no intervening cycle pushing moves to complicate things),
so this stack ends up in exactly the same configuration
as it was in $\rho$.
\end{proof}

Diagrammatically, writing $\rho' = \mathcal{C}\rho$, we have
\newarrow{Bothways}{<}{-}{-}{-}{>}
\begin{diagram}
\rho & \rTo^{\mathcal{C}} & \rho' \\
\dBothways < \Phi & & \dBothways > \Phi \\
\Phi(\rho) & \lTo^{\mathcal{C}} & \Phi(\rho')
\end{diagram}
Note the reversal of the direction of the $\mathcal{C}$ arrow.

\begin{lemma} \label{prop-flip-equivalent}
If $\rho \equiv \rho'$, then $\Phi(\rho) \equiv \Phi(\rho')$.
\end{lemma}

\begin{proof}
By Lemma~\ref{prop-equivalent}, 
it suffices to show that if two stack configurations 
are related by a cycle pushing move,
then their flips are related by a cycle pushing move.
But that is precisely what Lemma~\ref{prop-flip-reverse} tells us.
\end{proof}

\subsection{Antiparticles} \label{subsec-antiparticle}

Next we introduce antiparticles.  Like particles, they move from vertex to vertex in the graph,
but they interact with the stacks in a different way.
Suppose that the current stack configuration at $v$ is
\[ [ \dots, e_{-2}, e_{-1}, e_{0} | e_{1}, e_{2}, e_{3}, \dots ] \]
and that there is an antiparticle at $v$.
An antiparticle step consists of first
moving the particle along the arc $e_{0}$
and then pushing the stack at $v$ to obtain
\[ [ \dots, e_{-2}, e_{-1} | e_{0}, e_{1}, e_{2}, e_{3}, \dots ] . \]
(Compare: a particle step consists of first
popping the stack at $v$ to obtain 
\[ [ \dots, e_{-2}, e_{-1}, e_{0}, e_{1} | e_{2}, e_{3}, \dots ] \]
and then moving the particle along the arc $e_{1}$.)

\begin{lemma} \label{prop-antiparticle}
If $\rho'$ is obtained from $\rho$ 
by moving a particle from $v$ along arc $e$, 
then $\Phi(\rho')$ is obtained from $\Phi(\rho)$ 
by moving an antiparticle from $v$ along arc $e$.
\end{lemma}

\begin{proof}
Write the stack at $v$ for the rotor configuration $\rho$ as
\[ \rho_v = [\dots, e_{-2}, e_{-1}, e_{0} | e_{1}, e_{2}, e_{3}, \dots]. \]
When a particle at $v$ advances by one step, 
the particle moves along the arc $e_{1}$
and the stack at $v$ becomes
\[ \rho'_v = [\dots, e_{-2}, e_{-1}, e_{0}, e_{1} | e_{2}, e_{3}, \dots]. \]
On the other hand,
the stack at $v$ for the flipped rotor configuration $\Phi(\rho)$ is
\[ \Phi(\rho)_v = 
[\dots, e_{3}, e_{2}, e_{1} | e_{0}, e_{-1}, e_{-2}, \dots]. \]
When an antiparticle at $v$ advances by one step, 
the antiparticle moves along the arc $e_{1}$
and the stack at $v$ becomes
	\[ [\dots, e_{3}, e_{2} | e_{1}, e_{0}, e_{-1}, e_{-2}, \dots ] \]
which equals $\Phi(\rho')_v$.
\end{proof}

Just as one defines particle addition operators $E_v^{\ }$,
one can define antiparticle addition operators $E_v^{-}$ 
on rotor configurations: to apply $E_v^{-}$,
add an antiparticle at $v$ and let it perform rotor walk on $G$
(using the antiparticle dynamics described above)
until it arrives at a vertex in the target set $T$.
To highlight the symmetry between particles and antiparticles
we will sometimes write $E_v^{+}$ instead of $E_v^{\ }$
for particle addition operators.  Note that in general, $E_v^+$ and $E_w^-$ do not commute.

Write $t^+_v(\rho)$ (resp.\ $t^-_v(\rho)$) for the target vertex hit by a particle (resp.\ antiparticle) started at $v$ if the initial rotor configuration is $\rho$.

\old{
Example: Let $G$ be the graph with vertex set $\{0,1,2,3\}$
with targets 0 and 3,
with a rotor at 1 that alternates between 0 and 2
and a rotor at 2 that alternates between 1 and 3.
Adopting retrospective notation, we will write the four rotor states 
as $* 0 1 *$, $* 0 3 *$, $* 2 1 *$, and $* 2 3 *$,
where for instance $* 2 3 *$ denotes the stack configuration
in which the retrospective rotor states at 1 and 2
are 2 and 3, respectively.
(The $*$'s are meant to remind us 
of the presence of the target vertices 0 and 3 at the ends.)
The following table shows how 
the two particle addition operators $E_1^{+}$, $E_2^{+}$
and the two antiparticle addition operators $E_1^{-}$, $E_2^{-}$
act on the four rotor configurations:

\begin{center}
\begin{tabular}{l | l l l l}
        & $*01*$ & $*03*$ & $*21*$ & $*23*$ \\
\hline
$E_1^+$ & $*23*$ & $*01*$ & $*01*$ & $*03*$ \\
$E_2^+$ & $*03*$ & $*23*$ & $*23*$ & $*01*$ \\
$E_1^-$ & $*21*$ & $*23*$ & $*23*$ & $*01*$ \\
$E_2^-$ & $*23*$ & $*01*$ & $*01*$ & $*21*$
\end{tabular}
\end{center}

\noindent
Note that the particle addition operators 
and the antiparticle addition operators
do not commute with one another in this example.
}

\begin{lemma} \label{prop-switch}
For any rotor configuration $\rho$ and any $v \in V_0$ 
we have $\Phi(E_v^{+}(\rho)) = E_v^{-}(\Phi(\rho))$, and $t_v^+ (\rho) = t_v^-(\Phi(\rho))$.
\end{lemma}

\begin{proof}
This follows by repeated application
of Lemma~\ref{prop-antiparticle}:
the sequence of vertices traveled by the particle added to $\rho$ at $v$
is the same as the sequence of vertices 
traveled by the antiparticle added to $\Phi(\rho)$ at $v$.
\end{proof}

Diagrammatically, writing $\rho' = E_v^+ \rho$, we have:
\begin{diagram}
\rho & \rTo^{E_v^+} & \rho' \\
\dBothways < \Phi & & \dBothways > \Phi \\
\Phi(\rho) & \rTo^{E_v^-} & \Phi(\rho')
\end{diagram}

\begin{lemma}
\label{prop-anti-equivalent}
If $\rho \equiv \rho'$, then $E_v^- \rho \equiv E_v^- \rho'$ and $t_v^-(\rho) = t_v^-(\rho')$ for all $v \in V_0$.
\end{lemma}

\begin{proof}
We have $E_v^- = \Phi \circ E_v^+ \circ \Phi$ by Lemma~\ref{prop-switch}.  
Moreover, $\Phi$ preserves equivalence by Lemma~\ref{prop-flip-equivalent} 
and $E_v^+$ preserves equivalence by Lemma~\ref{prop-respect}, 
so $E_v^-$ must preserve equivalence.  This proves the first statement.  For the second, since $\Phi (\rho) \equiv \Phi (\rho')$, we have by Lemmas~\ref{prop-switch} and~\ref{prop-same-target}
	\[ t_v^-(\rho) = t_v^+ ( \Phi(\rho)) = t_v^+ (\Phi(\rho')) = t_v^-(\rho'). \qed \]
\renewcommand{\qedsymbol}{}
\end{proof}

\subsection{Loop-erasure} \label{subsec-loop}

If a path $(x_0,\dots,x_r)$ in the directed graph $G$
contains a cycle, 
i.e., a sub-path $(x_p,x_{p+1},\dots,x_q)$ with $x_q=x_p$,
define the {\it first cycle\/}
as the unique cycle with $q$ as small as possible;
we may replace the path by the shorter path
$(x_0,\dots,x_{p-1},x_p,x_{q+1},\dots,x_r)$
from which the $q-p$ vertices of the first cycle have been removed.
If this new path contains a cycle, 
we may erase the first cycle of the new path,
obtaining an even shorter path.
If we continue in this fashion,
we eventually obtain a simple path from $x_0$ to $x_r$,
called the {\it loop-erasure\/} of the original path.

The notion of loop-erasure is due to Lawler~\cite{Law80}, 
who studied the loop-erasure of random walk.  
As is mentioned at the end of  \textsection5 of \cite{HLMPPW},
there is also a connection between loop-erasure and rotor walk.
Given a rotor configuration $\rho$ and a set $S \subset V_0$,
define {\it popping $S$\/} as the operation
of popping the stack at each vertex in $S$ to obtain the new rotor configuration
\[ S^+ \rho (v) = \begin{cases} 
	\rho(v)^+ & \mbox{ if } v \in S, \\
	\rho(v) & \mbox{ otherwise.} \end{cases} \]
(Compare to cycle pushing \textsection\ref{subsec-push}, 
in which the rotors are regressed instead of progressed.)
For a rotor configuration $\rho$ and a vertex $v \in V_0$, let $\Gamma$ be the loop-erasure of the path $x_0, \ldots, x_r$ traveled by a particle performing rotor walk starting from $x_0 = v$ until it hits the target set.  Let $C_1, \ldots, C_m$ be the cycles erased to obtain $\Gamma$.  For any vertex $w$, the number of $i$ ($0 \leq i \leq r-1$) with $x_i=w$ is equal to the number of $j$ ($1 \leq j \leq m$)
for which $w \in C_j$, plus either 1 or 0 according to whether or not
$w \in \Gamma$.  Hence the final rotor configuration $E_v^+ \rho$ can be obtained from $\rho$ by popping the cycles $C_1,\ldots,C_m$ and the path $\gamma := \Gamma-\{x_r\}$; that is,
	\begin{equation} \label{e.pathandcyclepopping} E_v^+ \rho = \gamma^+ C_1^+ \ldots C_m^+ \rho. \end{equation}

\begin{lemma} \label{prop-loop-erase}
For every rotor configuration $\rho$ and every $v\in V_0$ we have
$E_v^{-} E_v^{+} \rho \equiv \rho$,
and the path traversed by the antiparticle
is the loop-erasure of the path traversed by the particle.
In particular, the antiparticle hits the same target as the particle:
	\[ t_v^- (E_v^+ \rho) = t_v^+ (\rho). \]
\end{lemma}

\begin{proof} 
After the particle has been added to $\rho$ at $v$,
changing the rotor configuration to $E_v^{+} \rho$
and arriving at target $t = t_v^+(\rho)$,
the retrospective rotor at each vertex $v$
is the arc that the particle traversed
the last time it left $v$.
Hence the rotors of $E_v^ + \rho$ give a simple (cycle-free) path $\gamma$
from $v$ to $t$,
and the antiparticle will travel this path,
arriving at the same target $t$.
By \eqref{e.pathandcyclepopping}, the rotor configuration $E_v^+ \rho$ is obtained from $\rho$ by a sequence of cycle-popping
moves followed by a ``path-popping move'' along $\gamma$.
The motion of the antiparticle from $v$ to $t$
undoes the path-popping move,
so all that survives in $E_v^- E_v^+ \rho$ are the cycle popping moves.
Since cycle popping doesn't change the equivalence class
of a rotor configuration (by Lemma~\ref{prop-equivalent}),
we conclude that $E_v^{-} E_v^{+} \rho \equiv \rho$.
\end{proof}

Likewise, for every $\rho$ we have $E_v^{+} E_v^{-} \rho \equiv \rho$.
Lemma~\ref{prop-loop-erase} thus says that the products
$E_v^+ E_v^-$ and $E_v^- E_v^+$
act as the identity operation on equivalence classes of rotor configurations.
That is, if we view $E_v^+$ and $E_v^-$ as elements of the sandpile group
(acting on equivalence classes of rotor configurations), they are inverses.

\subsection{Proof of the rotor-reversal theorem} \label{subsec-reversal}

Now we turn to the proof of our second main result, 
that reversal of the periodic pattern of the rotor mechanism at all vertices
causes reversal of the periodic pattern of the hitting sequence.  To save unnecessary notation in the proof, we write $E^{\pm} := E^{\pm}_s$ and $t^{\pm} := t^{\pm}_s$.

\begin{proof}[Proof of Theorem~\ref{thm-reversal}]

As in the proof of Theorem~\ref{thm-periodic}, 
the sequence of equivalence classes, $[\rho_0], [\rho_1], [\rho_2], \dots$ 
is periodic, say with period $D$.  
Now consider the hitting sequence for \emph{antiparticles} 
released from the source vertex $s$ from initial configuration $\rho_0$.  
Define $\eta_0 = \rho_0$ and $\eta_i = E^-(\eta_{i-1})$ for $i\geq 1$.  
We first show by induction on $i$ that 
$\eta_i \equiv \rho_{D-i}$ for all $i=0,\ldots,D$.  
The base case $i=0$ is the fact that $\rho_0 \equiv \rho_D$; 
and for $1 \leq i \leq D$, if $\eta_{i-1} \equiv \rho_{D-i+1}$ 
then by Lemmas~\ref{prop-anti-equivalent} and~\ref{prop-loop-erase},
	\[ \eta_{i} = E^-(\eta_{i-1}) \equiv  E^-(\rho_{D-i+1}) = E^- E^+ (\rho_{D-i}) \equiv \rho_{D-i} \]
which completes the inductive step.  

Now for $i \geq 1$, let $u_i^+ = t^+(\rho_{i-1})$ and $u_i^- = t^-(\eta_{i-1})$
be the hitting sequences for a particle (resp.\ antiparticle) started at $s$ 
with initial rotor configuration $\rho_0$.  
Using the second statements of Lemmas~\ref{prop-anti-equivalent} 
and~\ref{prop-loop-erase}, we have for $i=0,\ldots,D-1$
	\[ u^-_{i+1} = t^-(\eta_{i}) =  t^-(\rho_{D-i}) = t^-(E^+ \rho_{D-i-1})  = t^+(\rho_{D-i-1}) = u^+_{D-i}. \]
By Lemma~\ref{prop-switch}, the hitting sequence for a particle 
starting at $s$ with rotor configuration $\Phi(\rho_0)$ 
equals the hitting sequence for an \emph{antiparticle} 
starting at $s$ with rotor configuration $\rho_0$, 
that is, the sequence $\{u_i^-\}_{i\geq 1}$.  
Moreover, since $\eta_D = \rho_0 \equiv \rho_D = \eta_0$, 
the sequence $\{u^-_i\}_{i \geq 1}$ satisfies 
$u^-_{i+D} = u^-_{i}$ for all $i \geq 1$ by Lemma~\ref{prop-anti-equivalent}.  
Hence the particle hitting sequences for $\Phi(\rho_0)$ and $\rho_0$ 
are both periodic modulo $D$, 
and reversing the first $D$ terms of the latter hitting sequence 
yields the first $D$ terms of the former.
\end{proof}


\section*{Acknowledgments}

We thank Peter Winkler for launching this investigation
with his suggestion that the palindromic period-4 rotor
1,2,2,1,\dots would have special properties worthy of study.

\end{document}